\documentclass[openany,reqno,a4paper,11pt]{article}

\usepackage{amsmath,amssymb,amsthm}

\usepackage{mathrsfs}

\usepackage{mathtools}
\usepackage{commath}

\usepackage{thmtools}
\usepackage{thm-restate}

\usepackage{cases}
\usepackage{enumitem}
\setlist[enumerate,1]{label=(\roman*)}

\usepackage[pdftex, pdfborderstyle={/S/U/W 0}]{hyperref}
\hypersetup{
    colorlinks=true,
    linkcolor=magenta,
    citecolor=cyan,
}

\usepackage{cleveref}

\usepackage{etoolbox}

\usepackage{comment}

\usepackage[numbers]{natbib}

\numberwithin{equation}{section}

\ifdefined\thmcolor
\declaretheoremstyle[
  shaded={bgcolor=\thmcolor}
]{plain}
\else
\fi

\ifdefined\defcolor
\declaretheoremstyle[
  headfont=\normalfont\bfseries,
  bodyfont=\normalfont,
  shaded={bgcolor=\defcolor}
]{noital}
\else
\declaretheoremstyle[
  headfont=\normalfont\bfseries,
  bodyfont=\normalfont,
]{noital}
\fi

\declaretheorem[style=plain,numberwithin=section,name=Theorem]{theorem}

\declaretheorem[style=plain,sibling=theorem,name=Lemma]{lemma}

\declaretheorem[style=plain,sibling=theorem,name=Conjecture]{conjecture}
\declaretheorem[style=plain,sibling=theorem,name=Claim]{claim}

\declaretheorem[style=plain,sibling=theorem,name=Question]{question}

\declaretheorem[style=plain,numbered=no,name=Theorem]{theorem-n}
\declaretheorem[style=plain,numbered=no,name=Proposition]{proposition-n}
\declaretheorem[style=plain,numbered=no,name=Lemma]{lemma-n}
\declaretheorem[style=plain,numbered=no,name=Corollary]{corollary-n}
\declaretheorem[style=plain,numbered=no,name=Conjecture]{conjecture-n}
\declaretheorem[style=plain,numbered=no,name=Claim]{claim-n}
\declaretheorem[style=plain,numbered=no,name=Fact]{fact-n}
\declaretheorem[style=plain,numbered=no,name=Open Problem]{openproblem-n}
\declaretheorem[style=plain,numbered=no,name=Question]{question-n}
\declaretheorem[style=plain,numbered=no,name=Observation]{observation-n}

\declaretheorem[style=noital,sibling=theorem,name=Remark]{remark}
\declaretheorem[style=noital,sibling=theorem,name=Definition]{definition}

\declaretheorem[style=noital,numbered=no,name=Remark]{remark-n}
\declaretheorem[style=noital,numbered=no,name=Definition]{definition-n}
\declaretheorem[style=noital,numbered=no,name=Construction]{construction-n}
\declaretheorem[style=noital,numbered=no,name=Example]{example-n}

\newcommand{\defined}{\mathrel{\coloneqq}}

\newcommand{\st}{\mathbin{\colon}}

\undef{\set}
\DeclarePairedDelimiter{\set}{\lbrace}{\rbrace}

\undef{\emptyset}
\newcommand{\emptyset}{\varnothing}

\newcommand{\union}{\mathbin{\cup}}
\newcommand{\inter}{\mathbin{\cap}}

\newcommand{\from}{\colon}

\undef{\abs}
\DeclarePairedDelimiterX{\abs}[1]
  {\lvert}{\rvert}{\ifblank{#1}{\,\cdot\,}{#1}}

\undef{\norm}
\DeclarePairedDelimiterX{\norm}[1]
  {\lVert}{\rVert}{\ifblank{#1}{\,\cdot\,}{#1}}

\DeclarePairedDelimiterX{\inner}[2]
  {\langle}{\rangle}{\ifblank{#1}{\,\cdot\,}{#1},\ifblank{#2}{\,\cdot\,}{#2}}

\DeclareMathDelimiter{\given}
  {\mathbin}{symbols}{"6A}{largesymbols}{"0C}

\DeclareMathOperator{\Prob}{\mathbb{P}}
\DeclarePairedDelimiterXPP{\prob}[1]
  {\Prob}{\lparen}{\rparen}{}
  {\renewcommand{\given}{\nonscript\;\delimsize\vert\nonscript\;\mathopen{}}#1}

\DeclareMathOperator{\Expec}{\mathbb{E}}
\DeclarePairedDelimiterXPP{\expec}[1]
  {\Expec}{\lparen}{\rparen}{}
  {\renewcommand{\given}{\nonscript\;\delimsize\vert\nonscript\;\mathopen{}}#1}

\DeclareMathOperator{\Var}{Var}
\DeclarePairedDelimiterXPP{\var}[1]
  {\Var}{\lparen}{\rparen}{}
  {\renewcommand{\given}{\nonscript\;\delimsize\vert\nonscript\;\mathopen{}}#1}

\DeclareMathOperator{\Cov}{Cov}
\DeclarePairedDelimiterXPP{\cov}[2]
  {\Cov}{\lparen}{\rparen}{}{#1,#2}

\newcommand{\eps}{\varepsilon}
\newcommand{\sseq}{\subseteq}

\newcommand{\RR}{\mathbb{R}}

\newcommand{\ZZ}{\mathbb{Z}}

\newcommand{\cO}{\mathcal{O}}

\newcommand{\ga}{\alpha}

\usepackage{geometry}
\newcommand{\wavg}{\overline{w_M}}
\newcommand{\width}{\text{width}}

\begin{document}

\title{A uniform bound on almost colour-balanced perfect matchings in colour-balanced cliques}

\author{Lawrence Hollom \footnote{\href{mailto:lh569@cam.ac.uk}{lh569@cam.ac.uk}, Department of Pure Mathematics and Mathematical Statistics (DPMMS), University of Cambridge, Wilberforce Road, Cambridge, CB3 0WA, United Kingdom}}



\maketitle

\begin{abstract}
  An edge-colouring of a graph $G$ is said to be \emph{colour-balanced} if there are equally many edges of each available colour.
  We are interested in finding a colour-balanced perfect matching within a colour-balanced clique $K_{2nk}$ with a palette of $k$ colours.
  While it is not necessarily possible to find such a perfect matching, one can ask for a perfect matching as close to colour-balanced as possible. In particular, for a colouring $c:E(K_{2nk})\rightarrow [k]$, we seek to find a perfect matching $M$ minimising $f(M)\defined \sum_{i=1}^k\bigl||c^{-1}(i)\cap M|-n\bigr|$.

  The previous best upper bound, due to Pardey and Rautenbach, was $\min f(M)\leq \mathcal{O}(k\sqrt{nk\log k})$.
  We remove the $n$-dependence, proving the existence of a matching $M$ with $f(M)\leq 4^{k^2}$ for all $k$.
\end{abstract}


\section{Introduction}
\label{sec:intro}

The problem of finding colour-balanced perfect matchings can be considered a special case of zero-sum Ramsey theory, which has received significant study in recent years.
While work was initially concerned with zero-sum embeddings over finite groups, Caro and Yuster \cite{caro2015zero} initiated the study of zero-sum embeddings over $\ZZ$.
Several variants of this problem have since been studied; see for example \cite{caro2019bounded, caro2019k4}.

Embedding perfect matchings in this context was first considered by Caro, Hansberg, Lauri, and Zarb \cite{caro2020zero}.
They asked whether every two-edge-colouring of a clique on $4n$ vertices with equally many edges of each colour also contains a perfect matching with equally many edges of each colour.
This question was solved, affirmatively and independently, by Ehard, Mohr, and Rautenbach \cite{ehard2020perfect}, and by Kittipassorn and Sinsap \cite{kittipassorn2023perfect}.
Beyond solving this problem, Kittipassorn and Sinsap asked about the generalisation of this problem to more than two colours.
In particular, they \cite{kittipassorn2023perfect} asked whether every colour-balanced $k$ edge-colouring of the clique $K_{2kn}$ admits a colour-balanced perfect matching.
By calling a colouring \emph{colour-balanced}, we mean that there are equally many edges of each colour (as is standard), and we also use the notation $[k]\defined \set{1,2,\dotsc,k}$.

Pardey and Rautenbach \cite{pardey2022matchings} resolved this question in the negative, giving a colour-balanced three-colouring of $K_6$ which admits no perfect matching with one edge of each colour.
However, they also relaxed the question to ask merely for an almost colour-balanced perfect matching, and proposed the following conjecture.
\begin{conjecture}[{\cite[Conjecture 1]{pardey2022matchings}}]
\label{conj:pardey}
    If $n$ and $k$ are positive integers, and $c\from E(K_{2kn})\to [k]$ is colour-balanced, then there is a perfect matching $M$ of $K_{2kn}$ with
    \begin{align}
        f(M)\leq \cO(k^2),
    \end{align}
    where
    \begin{align}
        f(M)\defined \sum_{i=1}^k\Bigl|\abs{c^{-1}(i)\inter M}-n\Bigr|.
    \end{align}
\end{conjecture}

While they could only resolve their conjecture in the case of $k=3$, proving that there is always some $M$ with $f(M)\leq 2$, they could prove the following theorem.

\begin{theorem}[{\cite[Theorem 2]{pardey2022matchings}}]
\label{thm:pardey}
    For positive integers $n$ and $k$, and colour-balanced $c\from E(K_{2nk})\to [k]$, there is some perfect matching $M$ of $K_{2nk}$ satisfying
    \begin{align*}
        f(M)\leq 3k\sqrt{kn\log(2k)}.
    \end{align*}
\end{theorem}

While we cannot achieve the conjectured quadratic dependence on $k$, we take a significant step towards \Cref{conj:pardey} by removing the $n$-dependence from \Cref{thm:pardey}, proving the following uniform bound.

\begin{theorem}
\label{thm:main}
    For positive integers $n$ and $k$, and colour-balanced $c\from E(K_{2nk})\to [k]$, there is some perfect matching $M$ of $K_{2nk}$ satisfying
    \begin{align}
    \label{eq:target}
        f(M)\leq 4^{k^2}.
    \end{align}
\end{theorem}

One could also ask about finding colour-balanced perfect matchings in random colourings, and indeed this has been done; first by Frieze \cite{frieze2020random}, who produced results that were later extended by Chakraborti and Hasabnis \cite{chakraborti2021random}.
In short, much stronger results are known for the case of random colourings than in the deterministic case presented here. 
In particular, a random (not necessarily balanced) colouring of a clique contains a colour-balanced perfect matching with high probability.

Before providing a proof of \Cref{thm:main}, we first prove a result very similar to \Cref{thm:pardey} to serve as a warm-up, and introduce some of the ideas we will use in the main proof.


\section{A new proof of Theorem \ref{thm:pardey}}
\label{sec:warm-up}

In this section we prove the following theorem, which is slightly stronger than \Cref{thm:pardey}, but still has the crucial dependence on $\sqrt{n}$.
The purpose of \Cref{thm:warm-up} is not the result itself (as it is superseded by \Cref{thm:main}), but rather to introduce the ideas that will be used in the proof of our main result in \Cref{sec:reduction}.

\begin{theorem}
\label{thm:warm-up}
    If $n$ and $k$ are positive integers, and $c\from E(K_{2kn})\to [k]$ is colour-balanced, then there is a perfect matching $M$ of $K_{2kn}$ with
    \begin{align*}
        f(M)\leq k\sqrt{2n}.
    \end{align*}
\end{theorem}

Before giving the details of the proof, we present some discussion and preliminary definitions, which will be referred back to in later sections.

The idea in this proof is that, instead of seeking to reduce $f(M)$, the sum of the absolute deviations from balancedness, we seek to minimise the sum of the deviations squared.
To this end, we make the following definitions.
\begin{align}
\label{eq:define-w}
    w_M(e) \defined \abs{\set{e'\in M\st c(e')=c(e)}}.
\end{align}

In words, $w_M(e)$ is the number of edges in the matching $M$ with the same colour as the edge $e$.
We will refer to $w_M(e)$ as the \emph{weight} of the edge $e$.
We can then also define the function $w_M$ applied to a set $S\sseq E(K)$, and for convenience we also give a notation for the average value of $w_M$.
\begin{align}
    w_M(S)\defined \sum_{e\in S}w_M(e).\\
    \wavg(S)\defined w_M(S)/\abs{S}.
\end{align}

We make the following simple, but crucial, definition and observation.
\begin{align}
\label{eq:matching-sum}
    g(M)\defined w_M(M) = \sum_{e\in M} w_M(e) = \sum_{i=1}^k m_i(M)^2. 
\end{align}

Where $m_i(S)$ is the number of edges in the set $S$ of colour $i$.
We use equality \eqref{eq:matching-sum} to prove the following claim, which is the main tool in our proof of \Cref{thm:warm-up}.

\begin{claim}
\label{claim:swap}
    For edges $uv,xy\in M$, define $M'\defined (M-\set{uv,xy})\union\set{ux,vy}$.
    Then if $w_M(uv)+w_M(xy)-w_M(ux)-w_M(vy) > 4$, then $g(M') < g(M)$.
\end{claim}

\begin{proof}
    Removing edges $uv$ and $xy$ from $M$ decreases the value of $g(M)$ by at least $2(w_M(uv)+w_M(xy)-2)$, as we remove the edges with weights $w_M(uv)$ and $w_M(xy)$, and either reduce the weights of $w_M(uv)+w_M(xy)-2$ edges by 1 (when $c(uv)\neq c(xy)$), or reduce the weights of $w_M(uv)-2$ edges by 2 (when $c(uv)=c(xy)$).
    A similar calculation gives that adding edges $ux$ and $vy$ increases the value of $g(M)$ by at least $2(w_M(ux)+w_M(vy)-2)$, again with equality when the colours are equal, and distinct from the colours of $uv$ and $xy$.
    Thus we find that overall, 
    \begin{align*}
        g(M)-g(M')\geq 2(w_M(uv)+w_M(xy)-w_M(ux)-w_M(vy)-4),
    \end{align*}
    from which the claim follows.
\end{proof}

We will assume throughout the rest of the paper that the situation of \Cref{claim:swap} does not occur, and for this reason, we shall refer to $uv,xy\in M$ as in the statement of \Cref{claim:swap} as a \emph{contradicting swap}.
We are now ready to prove our warm-up theorem.

\begin{proof}[Proof of \Cref{thm:warm-up}]
    Throughout this proof, we let $K$ be our clique of order $2nk$.
    Assume that $M$ is some perfect matching minimising the value of $g(M)$.
    By symmetry, the average value of 
    $$2(w_M(uv)+w_M(xy))-w_M(ux)-w_M(uy)-w_M(vx)-w_M(vy)$$
    over all $uv,xy\in M$ is equal to $4(\wavg(M)-\wavg(E\setminus M))$.
    Thus if $\wavg(M)-\wavg(E\setminus M)>2$ there must be some $u,v,x,y\in V(K)$ forming a contradicting swap, whereupon \Cref{claim:swap} applies and we find some $M'$ with $g(M')<g(M)$, a contradiction.

    Therefore we may assume that $\wavg(M)-\wavg(E\setminus M)\leq 2$.
    Expanding definitions, this implies the following.
    \begin{align*}
        \frac{w_M(M)}{\abs{M}}\leq \frac{\sum_{e\in E\setminus M}w_M(e)}{\abs{E\setminus M}}+2.
    \end{align*}
    Substituting in that $g(M)=w_M(M)$, $\abs{M}=nk$, and $\abs{E}=\binom{2nk}{2}$, and using the fact that $E$ has $\abs{E}/k$ edges of each colour, we deduce the following.
    \begin{align*}
        \frac{1}{nk}g(M) \leq \frac{1}{\binom{2nk}{2}-nk}\sum_{i=1}^k\biggl(\frac{1}{k}\binom{2nk}{2}-m_i(M)\biggr) m_i(M) + 2.
    \end{align*}
    Noting that $\sum_{i=1}^k m_i(M)=\abs{M}=nk$ and $\binom{2nk}{2}-nk=2nk(nk-1)$, the above simplifies to
    \begin{align*}
        2(nk-1)g(M)\leq n(2nk-1)nk - g(M) + 4nk(nk-1).
    \end{align*}
    This in turn rearranges to
    \begin{align*}
        g(M)\leq n^2k+\frac{4nk(nk-1)}{2nk-1}.
    \end{align*}

    If we define the numbers $x_i$ by $m_i(M)=n+x_i$, so that $\sum_{i=1}^k x_i=0$, we see from \eqref{eq:matching-sum} that the above implies
    \begin{align*}
        \sum_{i=1}^k x_i^2 \leq \frac{4nk(nk-1)}{2nk-1} < 2nk.
    \end{align*}
    Then, by convexity, we see that, if $\sum_{i=1}^k x_i^2$ is fixed, $\sum_{i=1}^k\abs{x_i}$ is maximised when all $\abs{x_i}$ are equal (dropping the assumption that $\sum_{i=1}^k x_i = 0$).
    Therefore, we find that $\sum_{i=1}^k\abs{x_i}< k\sqrt{2n}$, as required.
\end{proof}


\section{Proof of Theorem \ref{thm:main}}
\label{sec:reduction}

We now extend the proof given in \Cref{sec:warm-up} to remove the $n$-dependence. 
The intuition here is that the equality case in the previous proof, namely that every swap decreases $w_M$ by 2, cannot actually occur.
In fact, as will be shown, it is very far removed from reality.
All notation and terminology from \Cref{sec:warm-up} will be used throughout this section as well.
As the cases $k=2,3$ have been dealt with (as discussed in \Cref{sec:intro}), we may assume that $k\geq 4$.

Our proof has several steps. 
Firstly, in \Cref{subsec:collecting-colours} we provide an algorithm for grouping colours into sets $A_1,\dotsc,A_t$, so that colours which occur a similar number of times in $M$ are in the same set.
Then in \Cref{subsec:partial-order}, we define a partial order $\succ$ on $[t]\times [t]$ to record in which cases swaps between colours in our sets would necessarily be contradicting (i.e. satisfying the conditions of \Cref{claim:swap}).

In \Cref{subsec:approximation-by-reals}, we demonstrate that $\succ$ is contained in a partial ordering found by labelling the $A_i$ with certain numbers $a_i$, and performing comparisons between the $a_i$ alone.
Finally in \Cref{subsec:contradiction} we combine the results of the previous subsections to derive a contradiction, completing the proof of \Cref{thm:main}.

Assume for contradiction that we have a perfect matching $M$ which minimises $g(M)$, but still has $f(M)>4^{k^2}$.


\subsection{Collecting similar colours}
\label{subsec:collecting-colours}

We iteratively construct sets of colours $A_1,\dotsc,A_t$ as follows.
Initially set $t=k$ and each $A_i=\set{i}$, assuming, after possibly reordering, that 
\begin{align}
\label{eq:set-ordering}
    m_1(M)\geq m_2(M)\geq\dots\geq m_k(M).
\end{align}
Then, if $\ell=k-t$ and for some $i$,
\begin{align}
\label{eq:combination-rules}
    \min\set{\abs{m_x(M)-m_y(M)}\st x\in A_i\text{ and }y\in A_{i+1}}\leq 4^{(\ell+1)k},    
\end{align}
then we combine sets $A_i$ and $A_{i+1}$ into one set.
We also decrease the index of sets $A_j$ with $j>i+1$ so that the indices remain consecutive and the ordering from \eqref{eq:set-ordering} is preserved.
The process terminates when either $t=1$, or no pair of sets satisfy \eqref{eq:combination-rules}.

Note that if the process terminates with $t=1$, i.e. all colours are collected into a single set, then we can bound the values of $m_i(M)$ as follows.
\begin{align*}
    \abs{m_i(M)-n}\leq m_1(M)-m_k(M) \leq \sum_{j=0}^{k-1}4^{jk} < 4^{k(k-1)+1}.
\end{align*}
From which \Cref{thm:main} quickly follows.

Thus we may assume that the process terminates with at least two distinct sets of colours remaining.
We now make some definitions and observations about the sets $A_i$, which we will refer back to later.

For any set $A\sseq[k]$, we define 
\begin{align*}
    \width(A)\defined\max\set{m_i(M)-m_j(M)\st i,j\in A}.
\end{align*}
Then we know that for all $i$,
\begin{align}
\label{eq:sum-of-widths}
    \sum_{i=1}^t\width(A_i)\leq \sum_{j=1}^\ell 4^{jk} < 2\bigl(4^{\ell k}\bigr).
\end{align}
Furthermore, as the algorithm for combining the $A_i$ halted, we know that if 
\begin{align*}
    d(A,B)\defined\min\set{\abs{m_i(M)-m_j(M)}\st i\in A, j\in B},
\end{align*}
then for all distinct $i,j$,
\begin{align}
\label{eq:distance-lower-bound}
    d(A_i,A_j) > 4^{(\ell+1)k}.
\end{align}
Finally, define $\ga(e)\in [t]$ to be such that $c(e)\in A_{\ga(e)}$ for any edge $e\in E$.
With these notions in hand, we proceed to the rest of the proof.


\subsection{Partially ordering pairs of colour sets}
\label{subsec:partial-order}

We now define a partial order $\succ$ on $[t]\times [t]$ as follows.
\begin{definition}
\label{def:partial-order}
    For $i,j,i',j'\in [t]$, let $(i,j)\succ (i',j')$ if every swap from colours $(x,y)\in A_i\times A_j$ to $(x',y')\in A_{i'}\times A_{j'}$ is contradicting, as in \Cref{claim:swap}.
    In other words, for all $x\in A_i,y\in A_j,x'\in A_{i'}, y'\in A_{j'}$, we have $w_M(x)+w_M(y)>w_M(x')+w_M(y')+4$.
\end{definition}

We now prove some results about the structure of the partial order $\succ$.
Firstly, define relations $\sim$ and $\simeq$ on $[t]\times [t]$ as follows.
Set $(i,j)\sim (i',j')$ if $(i,j)$ and $(i',j')$ are $\succ$-incomparable, and then let $\simeq$ be the transitive closure of $\sim$.
Finally, define $B_1,B_2,\dotsc,B_s$ to be the equivalence classes of the equivalence relation $\simeq$, ordered so that if $q<r$, $(i,j)\in B_q$ and $(i',j')\in B_r$, then $(i,j)\succ (i',j')$; it is natural to write 
\begin{align*}
    B_1\succ B_2\succ\dots\succ B_s.
\end{align*}
In particular, $(1,1)\in B_1$ and $(t,t)\in B_s$.

\begin{lemma}
\label{lem:s-is-big}
    For any $x,y,z\in [t]$ with $x\neq y$, we have that $(x,z)\not\simeq (y,z)$ and $(z,x)\not\simeq (z,y)$. In particular, $s\geq 2t-1$.
\end{lemma}

\begin{proof}
    The arguments in these two cases are entirely similar, so we consider the first case.
    Assume for contradiction that $(x,z)\simeq(y,z)$.
    Then there is some sequence 
    $$(x,z)=(a_0,b_0)\sim(a_1,b_1)\sim\dotsc\sim(a_r,b_r)=(y,z),$$
    and we may furthermore assume that we have chosen the triple $x,y,z$ and intermediate steps $(a_i,b_i)$ to minimise the value of $r$ amongst all choices with $x\neq y$ and $(x,z)\simeq (y,z)$.
    
    We see that $a_0,a_1,\dotsc,a_r,b_0,b_1,\dotsc,b_r$ are all distinct, as otherwise there would be a different choice of $x,y,z$ resulting in a smaller value of $r$.
    Thus by definition of $\sim$ and \eqref{eq:sum-of-widths}, we see that for any colours $b\in A_x$ and $c\in A_y$, we have the following.
    \begin{align*}
        \abs{m_b(M)-m_c(M)} &\leq \sum_{j=0}^r (4+\width(A_{a_j})+\width(A_{b_j})) \\
        &\leq 4(r+1) + \sum_{i=1}^t\width(A_i)\\
        &\leq 4(r+1) + \sum_{j=1}^\ell 4^{jk} \\
        &< 2\bigl(4^{\ell k}\bigr) \\
        &< 4^{(\ell+1)k}.
    \end{align*}
    Thus we see that the algorithm in \Cref{subsec:collecting-colours} should have collected $A_x$ and $A_y$ (and any sets between) into a single set, a contradiction, and the claim is proved.
\end{proof}

We collect the results of this subsection into the following remark.

\begin{remark}
\label{obs:no-down-arrows}
    We may assume that there are no edges $uv,xy\in M$, indices $i,j,i',j'\in [t]$, and $q,r\in [s]$ such that $q<r$, $(\ga(uv),\ga(xy))\in B_q$, and $(\ga(ux),\ga(vy))\in B_r$, as then $(i,j)\succ (i',j')$ and the swap from $uv, xy$ to $ux, vy$ would be contradicting.
\end{remark}


\subsection{Approximating \texorpdfstring{$\succ$ with $(\RR, +)$}{< with (R,+)}}
\label{subsec:approximation-by-reals}

In this subsection we prove the following lemma.
\begin{lemma}
    There are numbers $a_1,a_2,\dotsc,a_t\in\RR$ such that $a_i+a_j=a_{i'}+a_{j'}$ if and only if $(i,j)\simeq (i',j')$.
    Furthermore, we may assume that $a_1>a_2>\dots >a_t$.
\end{lemma}

\begin{proof}
    We will think of $a=(a_1,a_2,\dotsc,a_t)$ as a vector, and find it as an element of the null space of a certain matrix $N$.
    For each relation $(i,j)\simeq (i',j')$, we seek to enforce that $a_i+a_j=a_{i'}+a_{j'}$.
    To this end, add a row to $N$ with $+1$ in columns $i,j$, and $-1$ in columns $i',j'$, and 0 in all other columns (or a $+2$ if $i=j$ or $-2$ if $i'=j'$).
    We seek a vector $a=(a_1,a_2,\dotsc,a_t)$ such that $Na=0$.
    
    If we define 
    \begin{align*}
        b_i\defined\min\set{m_c(M)\st c\in A_i}\text{ for each }i\in [t], 
    \end{align*}
    then letting $b=(b_1,b_2,\dotsc,b_t)$, we have $Nb=\eps$, where $\eps=(\eps_1,\eps_2,\dotsc,\eps_t)$ is a vector of error terms.
    This vector $b$ is our starting point for finding $a$, but we still need to manipulate it to be exactly in the null space of $N$.
    For each $r\in [t]$, we have $\eps_r=b_i+b_j-b_{i'}-b_{j'}$ for some $(i,j)\simeq (i',j')$.
    Using an entirely similar argument to the proof of \Cref{lem:s-is-big}, we discover that
    \begin{align}
    \label{eq:eps=is=small}
        \eps_r < 2\bigl(4^{\ell k}\bigr)\text{ for all }r\in [t].
    \end{align}

    We now construct a $t\times t$ square matrix $T$ based on $N$.
    If $N$ has fewer than $t$ rows, then append additional all-zero rows to the end of $N$ to form the square matrix $T$.
    If $N$ has more than $t$ rows then, possibly reordering rows, we may assume that all rows of $N$ after the first $t$ are in the linear span (over $\RR$) of the first $t$ rows.
    Then we may let $T$ consist of the first $t$ rows of $N$.
    
    In either case, note that $Tb=\eps'$, where each component of $\eps'$ equal to either zero or some component of $\eps$; in particular, $\norm{\eps'}_\infty\leq\norm{\eps}_\infty$.

    If $T$ is the zero matrix, then we can let $a=b$.
    Otherwise, if $\lambda$ is the minimal nonzero eigenvalue of $T$, then we know by projecting $b$ onto the null space of $T$ that we can find some $a$ with $Ta=0$ and \begin{align*}
        \norm{a-b}_\infty \leq \abs{\lambda}^{-1}\norm{\eps'}_\infty.
    \end{align*}
    Next, note that the product of the nonzero eigenvalues of an integer matrix must itself be a nonzero integer (as it appears as a coefficient of the integer polynomial $p(x)=\det(T-x I)$), and hence is at least 1 in absolute value.
    Then, noting that the sum of the absolute value of the coefficients in any row of $N$ is 4, we see that 4 is an upper bound on the absolute value of any eigenvalue of $T$.
    Thus $4^{t-1}\abs{\lambda}\geq 1$, and so $\abs{\lambda} \geq 4^{1-t}$.
    Putting this all together, we discover the following.
    \begin{align*}
        \norm{a-b}_\infty &\leq \abs{\lambda}^{-1}(2(4^{\ell k})) \qquad \text{ by \eqref{eq:eps=is=small}} \\
        & \leq 4^t 4^{\ell k}/2 \\
        & < 4^{(\ell + 1)k}/8, \qquad \text{ as } k\geq 4\text{ and }k> t.
    \end{align*}
    Now we may further note that inequality \eqref{eq:distance-lower-bound} gives 
    \begin{align*}
        4^{(\ell + 1)k}/2 &< \min\set{d(A_i,A_{i+1})\st 1\leq i <t}/2 \\
        &\leq \min\set{b_i-b_{i+1}\st 1\leq i < t}/2.
    \end{align*}
    and so, as the sequence $b_i$ is strictly decreasing, we find that $a_i$ is strictly decreasing too, as required.
\end{proof}


\subsection{Deriving a contradiction}
\label{subsec:contradiction}

We now combine results from previous subsections together with a counting argument to produce a contradiction.
To begin, we make several definitions, each holding for all $i,j\in[t]$.

Define $S$ to be the set of ordered pairs of ordered edges in the matching $M$, up to flipping both edges.
$y_{i,j}$ is the number of swaps replacing some pair of edges in $M$ with a pair of edges with colours in $A_i$ and $A_j$.
\begin{align*}
    S&\defined\set{(uv,xy),(uv,yx),(xy,uv),(yx,uv)\st uv,xy\in M},\\
    y_{i,j}&\defined\abs{\set{(uv,xy)\in S \st \ga(ux)\in A_i,\ga(vy)\in A_j}}.
\end{align*}
Next, $p_i$ is the number of edges in $M$ of colours in $A_i$, and $p_{i,j}$ is a shorthand for the number of swaps wherein the edges removed from $M$ have colours in $A_i$ and $A_j$.
\begin{align*}
    p_i&\defined \sum_{c\in A_i} m_c(M),\\
    p_{i,j}&\defined\begin{cases}
        2p_ip_j & i\neq j\\
        2p_i(p_i-1) & i=j.
    \end{cases}
\end{align*}
Finally, $z_{i,j}$ measures the difference between $y_{i,j}$ and $p_{i,j}$; how many swaps start with colours in $A_i$ and $A_j$ against how many end there.
\begin{align*}
    z_{i,j}&\defined y_{i,j}-p_{i,j},\\
    \xi_i &\defined \sum_{j=1}^t z_{i,j}.
\end{align*}

Note that both $y_{i,j}$ and $p_{i,j}$ count pairs of edges in the matching $M$; equality \eqref{eq:y-and-p-sums} follows immediately from the definitions.
\begin{align}
\label{eq:y-and-p-sums}
    \sum_{i=1}^t\sum_{j=1}^t y_{i,j} = \sum_{i=1}^t\sum_{j=1}^t p_{i,j} = \abs{S} = 4\binom{\abs{M}}{2}.
\end{align}

And thus we also have the following.
\begin{align}
\label{eq:zero-sums}
    \sum_{i=1}^t\sum_{j=1}^t z_{i,j} = \sum_{i=1}^t \xi_i = 0.
\end{align}

Given the above, we now proceed to state and prove a claim concerning $z_{i,j}$.
\begin{claim}
\label{claim:positive-z-sums}
    For any $h\leq s$, the following holds.
    \begin{align*}
        \sum_{u=1}^h\sum_{(i,j)\in B_q}z_{i,j} \geq 0.
    \end{align*}
\end{claim}

\begin{proof}
    Assume for contradiction that for some $h$ the claimed inequality did not hold. 
    Unwrapping the definitions of $z_{i,j}$ and $y_{i,j}$, this would imply the following.
    \begin{align}
    \label{eq:deduce-contradicting-swap}
        \sum_{q=1}^h\sum_{(i,j)\in B_q}\abs{\set{(uv,xy)\in S \st \set{c(ux),c(vy)}=\set{i,j}}} <  \sum_{q=1}^h\sum_{(i,j)\in B_q}p_{i,j}.
    \end{align}
    Both sides of \eqref{eq:deduce-contradicting-swap} count elements of $S$.
    The RHS counts the number of ordered pairs of ordered edges $(uv,xy)\in S$ with colours $(c(uv),c(xy))\in B_1\union\dots\union B_h$, i.e. the number of swaps \emph{from} a pair of colours in $B_1\union\dots\union B_h$.
    The LHS of \eqref{eq:deduce-contradicting-swap} counts the number of such pairs for which $(c(ux),c(vy))\in B_1\union\dots\union B_h$, i.e. the number of swaps \emph{to} a pair of colours in $B_1\union\dots\union B_h$.
    Any swap from a pair of colours in $B_1\union\dots\union B_h$ to a pair of colours not in $B_1\union\dots\union B_h$ is necessarily contradicting, but \eqref{eq:deduce-contradicting-swap} tells us precisely that there are more of the former than the latter.
    
    Thus the inequality in \eqref{eq:deduce-contradicting-swap} implies that there must be a contradicting swap, a contradiction, and so the claim holds.
\end{proof}

We now manipulate the definition of $\xi_i$; our goal here is to show that $\xi_i/\abs{A_i}$ is an increasing function of $i$, from which we will derive our final contradiction. 
For the sake of intuition, not much is lost by considering the case when all $\abs{A_i}=1$, and this may ease understanding.

Indeed, in the case when all $A_i$ have size 1, $\sum_{j=1}^t y_{i,j}$ counts the number of edges of colour $i$ in $E\setminus M$, and $\sum_{j=1}^t p_{i,j}$ counts (a constant multiple of) the number of edges of colour $i$ in $M$, from which it is clear that $\xi_i$ should be increasing.
We now prove this generally.
We first deal with sums of $y_{i,j}$ and $p_{i,j}$.
\begin{align}
    \sum_{j=1}^t y_{i,j} &= \abs{\set{e\in E\setminus M\st \ga(e)=i}}=\abs{A_i}n(2nk-1)-p_i. 
\label{eq:y-partial-sum} \\
    \sum_{j=1}^t p_{i,j} &= \Bigl(\sum_{j=1}^t 2p_ip_j \Bigr) - 2p_i = 2p_i(nk-1). \label{eq:p-partial-sum}
\end{align}
Now we turn our attention to $\xi_i$.
\begin{align*}
    \frac{\xi_i}{\abs{A_i}} &= \sum_{j=1}^t \frac{z_{i,j}}{\abs{A_i}} \\
    &= \frac{1}{\abs{A_i}} \Bigl( \sum_{j=1}^t y_{i,j} - \sum_{j=1}^t p_{i,j} \Bigr) \\
    &= \frac{1}{\abs{A_i}}\bigl(\abs{A_i}n(2nk-1)-p_i - 2p_i(nk-1)\bigr) \qquad \text{ by \eqref{eq:y-partial-sum} and \eqref{eq:p-partial-sum}} \\
    &= n(2kn-1)-\frac{p_i(2kn-3)}{\abs{A_i}}.
\end{align*}

Noting that $p_i/\abs{A_i}$ is a strictly decreasing sequence, we see that $\xi_i/\abs{A_i}$ is strictly increasing.
To complete our proof, we consider the following value.
\begin{align*}
    \phi\defined\sum_{i=1}^t a_i\xi_i = \sum_{i=1}^t a_i\abs{A_i}\frac{\xi_i}{\abs{A_i}}.
\end{align*}

Define $\nu_i$ to be the sequence formed by $\abs{A_1}$ copies of $\xi_1/\abs{A_1}$, then $\abs{A_2}$ copies of $\xi_2/\abs{A_2}$, and so on.
Thus we may re-write $\phi$ as follows.
\begin{align*}
    \phi = \sum_{i=1}^k a_i\nu_i.
\end{align*}
Then we know that $a_i$ is a strictly decreasing sequence, $\nu_i$ is an weakly-increasing not-entirely-constant sequence with zero sum (due to \eqref{eq:zero-sums} and that $\xi_i/\abs{A_i}$ is strictly increasing), and so we see that $\phi<0$.
To reach our contradiction, we now prove the following claim.

\begin{claim}
    $\phi\geq 0$.
\end{claim}

\begin{proof}
    Consider the weaker situation where instead the numbers $z_{i,j}$ merely form some zero-sum symmetric matrix satisfying the conclusion of \Cref{claim:positive-z-sums}; we will show that this is already enough to deduce that $\phi\geq 0$.
    In the case when all $z_{i,j}=0$, we have $\phi=0$.

    Now, any other valid value of $z_{i,j}$ can be formed by starting in the case $z_{i,j}=0$, and then performing the following sequence of operations a number of times.
    \begin{itemize}
        \item Select $i\leq j,i'\leq j'$ with $(i,j)\in B_q$ and $(i',j')\in B_r$ for some $q\leq r$.
        \item Select some weight $\zeta> 0$, and replace $z_{i,j}$ with $z_{i,j}-\zeta$, and $z_{i',j'}$ with $z_{i',j'}+\zeta$.
        \item Similarly, move $\zeta$ weight from $z_{j,i}$ to $z_{j',i'}$ to preserve symmetry.
    \end{itemize}
    This sequence of operations decreases $\xi_i$ and $\xi_j$ by $\zeta$ and increases $\xi_{i'}$ and $\xi_{j'}$ by $\zeta$.
    But as $q\leq r$, we know that $a_i+a_j\leq a_{i'}+a_{j'}$, so the above increases (or keeps constant) the value of $\phi$.
    Thus $\phi\geq 0$, as required.
\end{proof}

This contradiction completes our proof of \Cref{thm:main}.


\section{Conclusion and future work}
\label{sec:conclusion}

We have removed the $n$ dependence from \Cref{thm:pardey} to prove \Cref{thm:main}, and thus shown that for a fixed number of colours, an arbitrarily large colour-balanced clique contains a perfect matching only a constant number of edges away from being colour balanced itself.
However, there is still a significant gap in the $k$-dependence between our result and the bound of \Cref{conj:pardey}.
The techniques presented here seem far from reaching a quadratic dependence on $k$, and so little effort has been expended in optimising the bound produced.
That said, the author believes that it would be interesting to attempt to improve the upper bound, and investigate whether quadratic dependence on $k$ is truly the correct answer.

One particular limitation of the techniques explored here is that they will only find a matching which locally minimises the function $g(M)$, and cannot directly say anything about the globally optimal matching.
Based on this observation, we ask the following question.

\begin{question}
    What is the largest value of $f(M)$ in terms of $k$ for which there exists a colouring of $K_{2kn}$ and a perfect matching $M$ for which any small perturbation of $M$ (i.e. performing a swap) increases the value of $f(M)$? 
    In particular, can this be significantly larger than the global minimum of $f$ over all perfect matchings $M$?
\end{question}

In terms of lower bounds, we have not discovered any instance of a colour-balanced clique $K$ with 
\begin{align*}
    \min\set{f(M)\st M\text{ a perfect matching of }K}>2.
\end{align*}
However, it is not too hard to see that there are balanced colourings and perfect matchings $M$ forming local minima for $g$ for which $f(M)=\Omega(k^2)$.
It is for this reason that we believe it is plausible that the cases of local and global minima are significantly different, in which case a new approach would be necessary in order to get close to the true bound.

A different direction to extend this work would be to consider almost colour-balanced $H$-factors; here we have dealt with the case $H=K_2$.
The natural extension in this direction is, in the author's opinion, the following question.

\begin{question}
    For which graphs $H$ on $r$ vertices and $m$ edges is it there a function $h_H\from \ZZ\to\ZZ$ with the following property?
    
    For any $n$ and any number $k$ of colours, any colour-balanced $c\from K_{rnk}\to [k]$ contains an $H$-factor $F$ satisfying
    \begin{align*}
        f(F)\leq h_H(k),
    \end{align*}
    where $f(F)\defined \sum_{i=1}^k\abs{\abs{c^{-1}(i)\inter F}-nm}$ measures how far $F$ is from being colour-balanced.
\end{question}

It would appear that approaching the above question requires ideas beyond those presented in this paper.


\section{Acknowledgement}

The author is grateful to have been supported by the Internal Graduate Studentship of Trinity College, Cambridge.


\bibliographystyle{abbrvnat}  
\renewcommand{\bibname}{Bibliography}
\bibliography{main}


\end{document}